\documentclass{article} \voffset=-.5in
\usepackage{a4wide}
\usepackage{xcolor}
\usepackage{amsmath,amssymb,amsthm}
\usepackage{enumitem}
\usepackage{mathtools}
\usepackage{tikz,tikz-cd}
\usepackage{hyperref}

\newtheorem{thm}{Theorem}[section]
\newtheorem{cor}[thm]{Corollary}

\newtheorem{lem}[thm]{Lemma}

\newtheorem*{thm*}{Theorem}
\newtheorem*{lem*}{Lemma}
\theoremstyle{definition}
\newtheorem{df}[thm]{Definition}
\newtheorem{ex}[thm]{Example}
\theoremstyle{remark}

\makeatletter
\g@addto@macro\bfseries{\boldmath}
\makeatother

\DeclareMathOperator{\Fix}{Fix}

\renewcommand{\iff}{\enskip\Leftrightarrow\enskip}

\newcommand{\Q}{\mathbb{Q}}
\newcommand{\Z}{\mathbb{Z}}

\newcommand{\R}{\mathbb{R}}

\renewcommand{\S}{\Sigma}

\title{\bf Non-affine $n$-valued maps on tori}

\author{
K.\ Dekimpe\thanks{Supported by Methusalem grant METH/21/03 -- long term structural funding of the Flemish Government.}, \;\; L.~De Weerdt\\
\small {KU Leuven Campus Kulak Kortrijk},
\small {8500 Kortrijk},
\small {Belgium} \\
\small {e-mail: Karel.Dekimpe@kuleuven.be},\;\;
\small {Lore.DeWeerdt@kuleuven.be}}

\begin{document}

\maketitle

\begin{abstract}
In this paper we construct $n$-valued maps on $k$-dimensional tori, where $n,k\geq 2$, that are not homotopic to affine $n$-valued maps. This is in high contrast with the single valued case, where any such map is homotopic to an affine (even linear) map. We do this by investigating necessary and sufficient algebraic conditions on certain induced morphisms.
\end{abstract}

\section{Preliminaries on $n$-valued maps}

Let $X$ be a connected compact manifold with fundamental group $\pi$ and universal cover $p:\tilde{X}\to X$. In what follows we will always view $\pi$ as the group of covering transformations of this universal covering. An \emph{n-valued map} $f:X\multimap X$ is a continuous (i.e.\ upper and lower semi-continuous) set-valued function, sending each point of $X$ to a set of $n$ distinct points in $X$. Equivalently, by the work of
Brown and Gon\c{c}alves \cite{browngoncalves}, 
the map $f$ can be viewed as a continuous single-valued function from $X$ to the \emph{unordered configuration space} \[
D_n(X)=\{\{x_1,\ldots,x_n\}\subseteq X\mid x_i\neq x_j \text{ if } i\neq j \},
\]
topologised as the quotient of the space $F_n(X)=\{(x_1,\ldots,x_n)\in X^n\mid x_i\neq x_j \text{ if } i\neq j \}$ under the action of the permutation group $\S_n$. The fixed point set of an $n$-valued map $f$ is \[
\Fix(f)=\{x\in X \mid x\in f(x) \}.
\]

In what follows we recall some of the results of 
Brown et al. \cite{charlotte},
where the study of fixed points of $f$ was done by considering  fixed points of lifts of $f$ to the so-called  \emph{orbit configuration space} \[
F_n(\tilde{X},\pi)=\{(\tilde{x}_1,\ldots,\tilde{x}_n)\in \tilde{X}^n \mid p(\tilde{x}_i)\neq p(\tilde{x}_j) \text{ if } i\neq j \},
\]
which is a covering space of $D_n(X)$ with covering map \[
p^n:F_n(\tilde{X},\pi)\to D_n(X):(\tilde{x}_1,\ldots,\tilde{x}_n)\mapsto \{p(\tilde{x}_1),\ldots,p(\tilde{x}_n)\}.
\]
The corresponding covering group is $\pi^n\rtimes \S_n$, which acts on $F_n(\tilde{X},\pi)$ by \[
(\gamma_1,\ldots,\gamma_n;\sigma)(\tilde{x}_1,\ldots,\tilde{x}_n)=(\gamma_1\tilde{x}_{\sigma^{-1}(1)},\ldots,\gamma_n\tilde{x}_{\sigma^{-1}(n)}).
\]
We also recall that the product in $\pi^n\rtimes \S_n$ is given by \[ 
(\alpha_1,\ldots,\alpha_n;\sigma) (\beta_1,\ldots,\beta_n;\tau) = (\alpha_1\beta_{\sigma^{-1}(1)},\ldots,\alpha_n\beta_{\sigma^{-1}(n)};\sigma\tau).
\]
A lift of $f$ can then be defined as a map $\tilde{f}:\tilde{X}\to F_n(\tilde{X},\pi)$ such that \[
\begin{tikzcd}
\tilde{X} \ar[r,"\tilde{f}"] \ar[d,"p"'] & F_n(\tilde{X},\pi) \ar[d,"p^n"] \\
X \ar[r,"f"] & D_n(X)
\end{tikzcd}
\]
commutes. As $F_n(\tilde{X},\pi)$ is a subspace of $\tilde{X}^n$, such a lift $\tilde{f}:\tilde{X}\to F_n(\tilde{X},\pi)$ splits into $n$ maps $\tilde{f}_1,\ldots,\tilde{f}_n:\tilde{X}\to\tilde{X}$, called the \emph{lift-factors} of $\tilde{f}$. 

A fixed chosen lift $\tilde{f}=(\tilde{f}_1,\ldots,\tilde{f}_n)$ induces a morphism $\psi_{\tilde{f}}=(\phi_1,\ldots,\phi_n;\sigma):\pi\to \pi^n\rtimes \S_n$ of the covering groups so that $\tilde{f} \circ \gamma = \psi_{\tilde{f}}(\gamma) \circ \tilde{f}$ for all $\gamma\in \pi$. Written out using the lift-factors, this becomes \begin{equation}\label{eq:psi}
(\tilde{f}_1\gamma,\ldots,\tilde{f}_n\gamma)=(\phi_1(\gamma)\tilde{f}_{\sigma_\gamma^{-1}(1)},\ldots,\phi_n(\gamma)\tilde{f}_{\sigma_\gamma^{-1}(n)}).
\end{equation}
The map $\sigma: \pi \to \Sigma_n$ will be a morphism of groups, while in general the maps $\phi_i: \pi \to \pi$ are not morphisms.

Using the morphism $\sigma$, one can subdivide $\{1,\ldots,n\}$ into \emph{$\sigma$-classes}, the equivalence classes for the relation given by \[
i\sim j \iff \exists \gamma\in\pi:\sigma_{\gamma}(i)=j.
\]
For all $i$, let $S_i=\{\gamma\in\pi\mid \sigma_{\gamma}(i)=i \}$ denote the stabilizer of $i$ for this relation, which is a finite index subgroup of $\pi$. The restrictions of the maps $\phi_i$ to $S_i$ are group morphisms, and one can use these to define an equivalence relation on $\pi$ by \[
\alpha\sim\beta \iff \exists \gamma\in S_i:\alpha=\gamma\beta\phi_i(\gamma)^{-1}.
\]
The set of equivalence classes for this relation is denoted $\mathcal{R}[\phi_i]$, and its number of elements is the \emph{Reidemeister number} $R(\phi_i)$.

The $\sigma$-classes and Reidemeister classes partition the fixed point set of $f$ into a disjoint union of \emph{fixed point classes}: if $i_1,\ldots,i_r$ are representatives of the $\sigma$-classes, then \[
\Fix(f)=\bigsqcup_{\ell=1}^r \bigsqcup_{[\alpha]\in \mathcal{R}[\phi_{i_\ell}]} p(\Fix(\alpha\tilde{f}_{i_\ell})).
\]
To each fixed point class one can associate an \emph{index}, and the fixed point classes with non-zero index are called \emph{essential}. The number of essential fixed point classes is the \emph{Nielsen number} of $f$, which is a lower bound for the minimal number of fixed points among all maps homotopic to $f$, \[
N(f)\leq \text{min}\{\# \Fix(g) \mid g\sim f \}.
\] 

At this point we would like to remark that when $f$ and $g$ are two homotopic $n$-valued maps, then for any lift $\tilde{f}$ of $f$, there exists a lift $\tilde{g}$ of $g$ such that  
$\psi_{\tilde{f}}=\psi_{\tilde{g}}$.

By the following result (see e.g.\ 
\cite{staecker}), we can restrict the study of Nielsen numbers of $n$-valued maps to \emph{irreducible} maps, whose image cannot be written as the union of the images of an $m$-valued map and an $(n-m)$-valued map, with $0<m<n$:

\begin{thm}
Given a map $f:X\to D_n(X)$, if $f_1:X\to D_m(X)$ and $f_2:X\to D_{n-m}(X)$ are maps such that $f=\{f_1,f_2\}$, the Nielsen number of $f$ is given by \[
N(f)=N(f_1)+N(f_2).
\]
\end{thm}

In \cite{staecker} it is also shown that a map $f:X\to D_n(X)$ is irreducible if and only if, for any choice of lift $\tilde{f}$, the induced morphism $\psi_{\tilde{f}}$ induces only one $\sigma$-class. Accordingly, we will call a general  morphism $\psi:\pi\to \pi^n\rtimes \S_n$ (even if it is not induced by an $n$-valued map) irreducible if it induces one $\sigma$-class.

In this paper we focus on the case where $X$ is the $k$-torus $T^k$, i.e. the $k$-fold product $S^1\times\ldots\times S^1$, with universal cover $p:\R^k\to T^k:(t_1,\ldots,t_k)\mapsto (e^{2\pi it_1},\ldots,e^{2\pi it_k})$ and fundamental group $\Z^k$. An interesting type of $n$-valued self-map of $T^k$ is the following.

\begin{df}
A map $f:T^k\to D_n(T^k)$ is called \emph{affine} if it has a lift $(\tilde{f}_1,\ldots,\tilde{f}_n):\R^k\to F_n(\R^k,\Z^k)$ each of whose lift-factors $\tilde{f}_i:\R^k\to \R^k$ is an affine map $\bar{t}\mapsto A_i\bar{t}+\bar{a}_i$.
Equivalently, it can be written as \[
f:T^k\to D_n(T^k):p(\bar{t})\mapsto \{p(A_1\bar{t}+\bar{a}_1),\ldots,p(A_n\bar{t}+\bar{a}_n)\}.
\]
\end{df}

The interest in affine maps lies in the fact that their Reidemeister and Nielsen numbers are very easy to compute in terms of the matrices $A_i$:

\begin{thm}\label{thm:nv-RN}
If $f:T^k\to D_n(T^k):p(\bar{t})\mapsto \{p(A_1\bar{t}+\bar{a}_1),\ldots,p(A_n\bar{t}+\bar{a}_n)\}$ is an affine $n$-valued map, then
\[
R(f)=\sum_{i=1}^n|\det(I_k-A_i)|_\infty \quad \text{and} \quad
N(f)=\sum_{i=1}^n|\det(I_k-A_i)|,
\]
where $I_k\in \R^{k\times k}$ denotes the identity matrix.
\end{thm}

Moreover, in the single-valued case $n=1$, we have the following.

\begin{thm}
For a map $f:T^k\to T^k$, let $A\in \Z^{k\times k}$ be the matrix such that the induced fundamental group morphism $f_*:\Z^k\to \Z^k$ is given by multiplication with $A$. Then $f$ is homotopic to the affine map $T^k\to T^k:p(\bar{t})\mapsto p(A\bar{t})$. Therefore the Reidemeister and Nielsen numbers of $f$ are \[
R(f)=|\det(I_k-A)|_\infty \quad \text{and} \quad
N(f)=|\det(I_k-A)|.
\]
\end{thm}

In fact, more generally it can be shown that any single-valued map on a so-called \emph{infra-nilmanifold} (of which tori are a special case) is homotopic to an affine map, and one has also established formulas for the Reidemeister and Nielsen numbers for such affine maps that generalise the ones above (\cite{felshtyn-lee, leelee, kblee}).

For $n$-valued maps, the only case that is completely understood so far is the $1$-dimensional case of the circle $S^1$. In \cite{brown2006}, Brown showed that any $n$-valued map on $S^1$ is homotopic to an affine map, so that the Reidemeister and Nielsen numbers of all $n$-valued maps on $S^1$ are determined by Theorem \ref{thm:nv-RN}.

In general, for $n\geq 2$ and $k\geq 2$, it is no longer true that any $n$-valued map of $T^k$ is homotopic to an affine map. The goal of this paper is to find conditions on the induced morphism $\psi_{\tilde{f}}$ for $f$ to be homotopic to an affine map. More generally, we ask:

\begin{center}
\itshape
Given a morphism $\psi:\Z^k\to (\Z^k)^n\rtimes \S_n$, is there an affine map \\ $f:T^k\to D_n(T^k)$ such that $\psi_{\tilde{f}}=\psi$, for some lift $\tilde{f}$ of $f$?
\end{center}

As observed above, in order to study fixed points (Nielsen numbers) of $n$-valued maps, it suffices to study irreducible morphisms $\psi$. In that case, we will give a necessary and sufficient condition for $\psi$ to be induced by an affine map. Up to a smart choice of lift, this condition reduces to a more concrete one, which can be used to construct examples of non-affine morphisms $\psi$ and corresponding non-affine torus maps.

\section{The divisibility condition}

Although it suffices to consider irreducible morphisms, most results in this paper remain valid in the general case. Therefore we will restrict to irreducible morphisms only when strictly necessary.

Consider an arbitrary morphism $\psi=(\phi_1,\ldots,\phi_n;\sigma):\Z^k\to (\Z^k)^n\rtimes \S_n$. Note that, even though $\psi$ need not be induced by an $n$-valued map, it still makes sense to talk about $\sigma$-classes and the groups $S_i=\{\bar{z}\in \Z^k \mid \sigma_{\bar{z}}(i)=i \}$ as above. Then we have the following.

\begin{lem}
Let $\psi=(\phi_1,\ldots,\phi_n;\sigma):\Z^k\to (\Z^k)^n\rtimes \S_n$ be a morphism.
\begin{itemize}[nolistsep]
\item[\rm(i)] The maps $\phi_i:\Z^k\to \Z^k$ satisfy $\phi_i(\bar{z}_1+\bar{z}_2)=\phi_i(\bar{z}_1)+\phi_{(\sigma_{\bar{z}_1})^{-1}(i)}(\bar{z}_2)$ for all $i,\bar{z}_1,\bar{z}_2$. In particular, for all $i$, the restriction of $\phi_i$ to $S_i$ is a group morphism.
\medskip
\item[\rm(ii)] The map $\sigma:\Z^k\to \S_n$ satisfies $\sigma_{\bar{z}_1+\bar{z}_2}=\sigma_{\bar{z}_1}\sigma_{\bar{z}_2}$ for all $\bar{z}_1,\bar{z}_2$. In particular, $\sigma_{\bar{z}_1}\sigma_{\bar{z}_2}=\sigma_{\bar{z}_2}\sigma_{\bar{z}_1}$ for all $\bar{z}_1,\bar{z}_2$.
\end{itemize}
\end{lem}

\begin{proof}
Let $\bar{z}_1,\bar{z}_2\in \Z^k$ be arbitrary. Then \begin{align*}
 \psi(\bar{z}_1+\bar{z}_2)&=(\phi_1(\bar{z}_1+\bar{z}_2),\ldots,\phi_n(\bar{z}_1+\bar{z}_2);\sigma_{\bar{z}_1+\bar{z}_2}) \\
\psi(\bar{z}_1)\psi(\bar{z}_2) &= (\phi_1(\bar{z}_1)+\phi_{(\sigma_{\bar{z}_1})^{-1}(1)}(\bar{z}_2),\ldots,\phi_n(\bar{z}_1)+\phi_{(\sigma_{\bar{z}_1})^{-1}(n)}(\bar{z}_2);\sigma_{\bar{z}_1}\sigma_{\bar{z}_2}).
\end{align*}
Since $\psi$ is a group morphism, these must be equal. Viewing the equality coordinate by coordinate gives the result.
\end{proof}

For all $i$ and $\bar{z}$, denote by $n_{i\bar{z}}$ the length of the cycle in the disjoint cycle decomposition of $\sigma_{\bar{z}}$ containing $i$. Note that $n_{i\bar{z}}\bar{z}\in S_i$, since $\sigma_{n_{i\bar{z}}\bar{z}}(i)=(\sigma_{\bar{z}})^{n_{i\bar{z}}}(i)=i$. Additionally,

\begin{lem}\label{lem:n_iz-phi_i}
If $i$ and $j$ belong to the same $\sigma$-class, then
\begin{itemize}[nolistsep]
\item[\rm(i)] $S_i=S_{j}$
\item[\rm(ii)] $n_{i\bar{z}}=n_{j\bar{z}}$ for all $\bar{z}\in \Z^k$
\item[\rm(iii)] $\phi_i(\bar{z})=\phi_{j}(\bar{z})$ for all $\bar{z}\in S_i=S_{j}$.
\end{itemize}
\end{lem}

\begin{proof}
Take $\bar{z}_0\in\Z^k$ such that $j=(\sigma_{\bar{z}_0})^{-1}(i)$.

For (i), suppose $\bar{z}\in S_i$, so $i=\sigma_{\bar{z}}(i)$. Applying $({\sigma_{\bar{z}_0}})^{-1}$ to both sides yields \[
({\sigma_{\bar{z}_0}})^{-1}(i)=({\sigma_{\bar{z}_0}})^{-1}\sigma_{\bar{z}}(i)=\sigma_{\bar{z}}({\sigma_{\bar{z}_0}})^{-1}(i).
\] 
That is, $j=\sigma_{\bar{z}}(j)$, which means $\bar{z}\in S_{j}$. The converse inclusion follows by interchanging the roles of $i$ and $j$.

For (ii), notice that for all $\bar{z}\in \Z^k$, \[
(\sigma_{\bar{z}})^{n_{j\bar{z}}}(i)=(\sigma_{\bar{z}})^{n_{j\bar{z}}}\sigma_{\bar{z}_0}(j)=\sigma_{\bar{z}_0}(\sigma_{\bar{z}})^{n_{j\bar{z}}}(j)=\sigma_{\bar{z}_0}(j)=i,
\]
so we have $n_{i\bar{z}}\leq n_{j\bar{z}}$. Interchanging the roles of $i$ and $j$ gives the converse inequality.

For (iii), suppose $\bar{z}\in S_i$. Then we have \begin{align*}
\phi_i(\bar{z})+\phi_i(\bar{z}_0)&=\phi_i(\bar{z})+\phi_{(\sigma_{\bar{z}})^{-1}(i)}(\bar{z}_0) \\
&=\phi_i(\bar{z}+\bar{z}_0) \\&=\phi_i(\bar{z}_0)+\phi_{(\sigma_{\bar{z}_0})^{-1}(i)}(\bar{z})\\&=\phi_i(\bar{z}_0)+\phi_{j}(\bar{z}).
\end{align*}
It follows that $\phi_i(\bar{z})=\phi_{j}(\bar{z})$.
\end{proof}

\begin{lem}\label{lem:aff-wd}
If $f:T^k\to D_n(T^k)$ is an affine map with lift $\tilde{f}:\bar{t}\mapsto (A_1\bar{t}+\bar{a}_1,\ldots,A_n\bar{t}+\bar{a}_n)$ such that $\psi_{\tilde{f}}=(\phi_1,\ldots,\phi_n;\sigma)$, the maps $\phi_i$ and $\sigma$ satisfy \[
\left\{
\begin{array}{l}
A_i\bar{z}+\bar{a}_i=\bar{a}_{(\sigma_{\bar{z}})^{-1}(i)}+\phi_i(\bar{z}) \\
A_i=A_{(\sigma_{\bar{z}})^{-1}(i)}.
\end{array}
\right.
\]
for all $i$ and $\bar{z}$.
\end{lem}

\begin{proof}
For an affine map with lift $\tilde{f}$ as above, equation (\ref{eq:psi}) in coordinate $i$ reduces to \[
A_i\bar{t}+A_i\bar{z}+\bar{a}_i=A_{(\sigma_{\bar{z}})^{-1}(i)}\bar{t}+\bar{a}_{(\sigma_{\bar{z}})^{-1}(i)}+\phi_i(\bar{z}).
\]
Filling in $\bar{t}=\bar{0}$ implies \[
A_i\bar{z}+\bar{a}_i=\bar{a}_{(\sigma_{\bar{z}})^{-1}(i)}+\phi_i(\bar{z}).
\]
Plugging this back into the previous equation then yields $A_i\bar{t}=A_{(\sigma_{\bar{z}})^{-1}(i)}\bar{t}$ for all $\bar{t}\in \R^k$. 
Hence the matrices $A_i$ and $A_{(\sigma_{\bar{z}})^{-1}(i)}$ must be equal.
\end{proof}

\begin{lem}\label{lem:nv-ai-diff}
If a morphism $\psi=(\phi_1,\ldots,\phi_n;\sigma):\Z^k\to (\Z^k)^n\rtimes \S_n$ is induced by an affine map $f:T^k\to D_n(T^k)$ with lift $\tilde{f}:\bar{t}\mapsto (A_1\bar{t}+\bar{a}_1,\ldots,A_n\bar{t}+\bar{a}_n)$, then $\bar{a}_i-\bar{a}_{(\sigma_{\bar{z}})^{-1}(i)}\notin\Z^k$ whenever $(\sigma_{\bar{z}})^{-1}(i)\neq i$.
\end{lem}

\begin{proof}
Suppose $(\sigma_{\bar{z}})^{-1}(i)\neq i$. Since $f$ is $n$-valued, we have \[ 
p(A_i\bar{t}+\bar{a}_i)\neq p(A_{(\sigma_{\bar{z}})^{-1}(i)}\bar{t}+\bar{a}_{(\sigma_{\bar{z}})^{-1}(i)}).
\]
Since $A_i=A_{(\sigma_{\bar{z}})^{-1}(i)}$ by the previous lemma, it follows that $\bar{a}_i-\bar{a}_{(\sigma_{\bar{z}})^{-1}(i)}\notin\Z^k$.
\end{proof}

The condition $\bar{a}_i-\bar{a}_{(\sigma_{\bar{z}})^{-1}(i)}\notin\Z^k$ can be rewritten purely in terms of the morphism $\psi$:

\begin{lem}\label{lem:ai-diff-phi}
Let $\psi=(\phi_1,\ldots,\phi_n;\sigma):\Z^k\to (\Z^k)^n\rtimes \S_n$ be an arbitrary morphism. If $A_i\in \R^{k\times k}$ and $\bar{a}_i\in \R^k$ are matrices and points such that $A_i\bar{z}+\bar{a}_i=\bar{a}_{(\sigma_{\bar{z}})^{-1}(i)}+\phi_i(\bar{z})$ for all $i$ and $\bar{z}$, then some given $i$ and $\bar{z}$ satisfy $\bar{a}_i-\bar{a}_{(\sigma_{\bar{z}})^{-1}(i)}\in\Z^k$ if and only if $\phi_i(n_{i\bar{z}}\bar{z})\in n_{i\bar{z}}\Z^k$.
\end{lem}

\begin{proof}
Since $\phi_i(\bar{z})\in \Z^k$, the condition $A_i\bar{z}+\bar{a}_i=\bar{a}_{(\sigma_{\bar{z}})^{-1}(i)}+\phi_i(\bar{z})$ implies that $\bar{a}_i-\bar{a}_{(\sigma_{\bar{z}})^{-1}(i)}\in\Z^k$ if and only if $A_i\bar{z}\in\Z^k$. On the other hand, evaluating at $n_{i\bar{z}}\bar{z}\in S_i$ gives $A_i\,n_{i\bar{z}}\bar{z}=\phi_i(n_{i\bar{z}}\bar{z})$, so $A_i\bar{z}=\phi_i(n_{i\bar{z}}\bar{z})/n_{i\bar{z}}$; and 
hence $A_i\bar{z}\in\Z^k$ if and only if $\phi_i(n_{i\bar{z}}\bar{z})\in n_{i\bar{z}}\Z^k$.
\end{proof}

As a consequence, we can now formulate a necessary condition for a morphism $\psi$ to be induced by an affine map.

\begin{thm}\label{thm:nec-cond}
Given a morphism $\psi=(\phi_1,\ldots,\phi_n;\sigma):\Z^k\to (\Z^k)^n\rtimes \S_n$, let the sets $S_i$ and the numbers $n_{i\bar{z}}$ be as above. If $\phi_i(n_{i\bar{z}}\bar{z})\in n_{i\bar{z}}\Z^k$ for some $i$ and $\bar{z}$ such that $\bar{z}\notin S_i$, then $\psi$ cannot be induced by an affine $n$-valued map.
\end{thm}

\begin{proof}
Suppose an affine map with lift $\tilde{f}:\bar{t}\mapsto (A_1\bar{t}+\bar{a}_1,\ldots,A_n\bar{t}+\bar{a}_n)$ induces $\psi$, and let $i$ and $\bar{z}$ be as above. Then Lemma \ref{lem:aff-wd} implies $A_i\bar{z}+\bar{a}_i=\bar{a}_{(\sigma_{\bar{z}})^{-1}(i)}+\phi_i(\bar{z})$, so Lemma \ref{lem:ai-diff-phi} implies $\bar{a}_i-\bar{a}_{(\sigma_{\bar{z}})^{-1}(i)}\in\Z^k$, which contradicts Lemma \ref{lem:nv-ai-diff} since $\bar{z}\notin S_i$.
\end{proof}

To obtain the converse of the above result, we restrict to irreducible morphisms $\psi$. In the irreducible case, it follows from Lemma \ref{lem:n_iz-phi_i} that the groups $S_i$, the numbers $n_{i\bar{z}}$ and the morphisms $\phi_i:S_i\to \Z^k$ do not depend on $i$; therefore we will denote them by $S$, $n_{\bar{z}}$ and $\phi:S\to \Z^k$, respectively. Note that outside of $S$, the maps $\phi_i$ may still differ.

\begin{thm}\label{thm:suff-cond}
If $\psi=(\phi_1,\ldots,\phi_n;\sigma):\Z^k\to (\Z^k)^n\rtimes \S_n$ is an irreducible morphism, with $S$, $n_{\bar{z}}$ and $\phi$ as above, then $\psi$ can be induced by an affine $n$-valued map if and only if $\phi(n_{\bar{z}}\bar{z})\notin n_{\bar{z}}\Z^k$ whenever $\bar{z}\notin S$.
\end{thm}

\begin{proof}
The `only if' part is the content of Theorem \ref{thm:nec-cond}. 
To prove the `if' part, we construct an explicit affine $n$-valued map $f$ inducing the morphism $\psi$.

Since $\phi:S\to \Z^k$ is a morphism from a finite index subgroup of $\Z^k$ to $\Z^k$, it is given by multiplication with a matrix $A\in \Q^{k\times k}$. Notice that $A$ has coefficients in $\Q$ instead of $\Z$ since the image of $\phi$ is not necessarily contained in $S$.

On the other hand, we can define points $\bar{a}_1,\ldots,\bar{a}_n\in \R^k$ by setting \[
\bar{a}_{(\sigma_{\bar{z}})^{-1}(1)}=A\bar{z}-\phi_1(\bar{z})
\]
for all $\bar{z}$.
This determines all points $\bar{a}_i$ since any $i$ can be written as $(\sigma_{\bar{z}})^{-1}(1)$ for some $\bar{z}$, because $\psi$ is irreducible. If $(\sigma_{\bar{z}_1})^{-1}(1)=(\sigma_{\bar{z}_2})^{-1}(1)$ for different values $\bar{z}_1$ and $\bar{z}_2$, we obtain the same result for $\bar{a}_{(\sigma_{\bar{z}_1})^{-1}(1)}$ and $\bar{a}_{(\sigma_{\bar{z}_2})^{-1}(1)}$: since $\sigma$ is a morphism, we have $(\sigma_{\bar{z}_2-\bar{z}_1})^{-1}(1)=1$, so $\bar{z}_2-\bar{z}_1\in S$; and therefore $A(\bar{z}_2-\bar{z}_1)=\phi(\bar{z}_2-\bar{z}_1)$ and \begin{align*}
\bar{a}_{(\sigma_{\bar{z}_1})^{-1}(1)}&=A\bar{z}_1-\phi_1(\bar{z}_1) \\
&=A\bar{z}_2-\phi_1(\bar{z}_1)-\phi(\bar{z}_2-\bar{z}_1) \\
&=A\bar{z}_2-\phi_1(\bar{z}_1)-\phi_{(\sigma_{\bar{z}_1})^{-1}(1)}(\bar{z}_2-\bar{z}_1) \\
&=A\bar{z}_2-\phi_1(\bar{z}_1+(\bar{z}_2-\bar{z}_1)) \\
&=A\bar{z}_2-\phi_1(\bar{z}_2) \\
&=\bar{a}_{(\sigma_{\bar{z}_2})^{-1}(1)}.
\end{align*}
Also, these points $\bar{a}_i$ satisfy $A\bar{z}+\bar{a}_i=\bar{a}_{(\sigma_{\bar{z}})^{-1}(i)}+\phi_i(\bar{z})$ for all $i$ and $\bar{z}$. Indeed, write $i=(\sigma_{\bar{z}_0})^{-1}(1)$; then \begin{align*}
\bar{a}_{(\sigma_{\bar{z}})^{-1}(i)}-\bar{a}_i&=\bar{a}_{(\sigma_{\bar{z}+\bar{z}_0})^{-1}(1)}-\bar{a}_{(\sigma_{\bar{z}_0})^{-1}(1)} \\
&=(A(\bar{z}+\bar{z}_0)-\phi_1(\bar{z}+\bar{z}_0))-(A\bar{z}_0-\phi_1(\bar{z}_0)) \\
&=A\bar{z}+A\bar{z}_0-\phi_1(\bar{z}_0)-\phi_{(\sigma_{\bar{z}_0})^{-1}(1)}(\bar{z})-A\bar{z}_0+\phi_1(\bar{z}_0) \\
&=A\bar{z}-\phi_{(\sigma_{\bar{z}_0})^{-1}(1)}(\bar{z}) \\
&=A\bar{z}-\phi_i(\bar{z}).
\end{align*}
By Lemma \ref{lem:ai-diff-phi}, the condition that $\phi(n_{\bar{z}}\bar{z})\notin n_{\bar{z}}\Z^k$ whenever $\bar{z}\notin S$ implies that $\bar{a}_i-\bar{a}_{(\sigma_{\bar{z}})^{-1}(i)}\notin\Z^k$ whenever $(\sigma_{\bar{z}})^{-1}(i)\neq i$. Since $\psi$ is irreducible, this means that $\bar{a}_i-\bar{a}_{j}\notin\Z^k$ for any distinct $i$ and $j$, since there is always a $\bar{z}\in \Z^k$ such that $j=(\sigma_{\bar{z}})^{-1}(i)$. In particular, $p(A\bar{t}+\bar{a}_i)\neq p(A\bar{t}+\bar{a}_{j})$ for all $i\neq j$ and $\bar{t}$. It follows that this construction determines a well-defined affine $n$-valued map \[
f:T^k\to T^k:\bar{t}\mapsto \{p(A\bar{t}+\bar{a}_1),\ldots,p(A\bar{t}+\bar{a}_n) \}.
\] 
Since $A\bar{z}+\bar{a}_i=\bar{a}_{(\sigma_{\bar{z}})^{-1}(i)}+\phi_i(\bar{z})$ for all $i$ and $\bar{z}$, this map $f$ induces the morphism $\psi$ with respect to the choice of lift $\bar{t}\mapsto (A\bar{t}+\bar{a}_1,\ldots,A\bar{t}+\bar{a}_n)$.
\end{proof}

Note that the condition $\phi(n_{\bar{z}}\bar{z})\notin n_{\bar{z}}\Z^k$ is only used to prove that $\bar{a}_i-\bar{a}_{(\sigma_{\bar{z}})^{-1}(i)}\notin\Z^k$ whenever $(\sigma_{\bar{z}})^{-1}(i)\neq i$. Therefore it suffices to check this condition for one representative vector $\bar{z}$ for each value $(\sigma_{\bar{z}})^{-1}(i)$. For instance, if $\bar{e}_1,\ldots,\bar{e}_k$ is a basis for $\Z^k$ and we denote $\sigma_{\bar{e}_j}=\vcentcolon \sigma_j$ and $n_{\bar{e}_j}=\vcentcolon n_j$ for all $j\in \{1,\ldots,k\}$, it suffices to consider the vectors $\bar{z}$ in the finite set \[
\left\{m_1\bar{e}_1+\ldots+m_k\bar{e}_k \mid m_j=0,\ldots,n_j-1 \right\},
\]
since for any $\bar{z}=z_1\bar{e}_1+\ldots+z_k\bar{e}_k$, \[
\sigma_{\bar{z}}={\sigma_1}^{z_1}\cdots {\sigma_k}^{z_k}={\sigma_1}^{z_1\,\text{mod}\,n_1}\cdots {\sigma_k}^{z_k\,\text{mod}\,n_k}.
\]
This turns Theorem \ref{thm:suff-cond} into an explicit algorithm to check whether a given irreducible morphism $\psi$ can be induced by an affine map.

\section{The cycle condition}

Theorem \ref{thm:nec-cond} gives rise to the following more concrete necessary condition for $\psi$ to be induced by an affine map, which is useful for constructing examples of $n$-valued maps that are not homotopic to an affine $n$-valued map.

\begin{cor}\label{cor:phis-equal}
Suppose a morphism $\psi=(\phi_1,\ldots,\phi_n;\sigma):\Z^k\to (\Z^k)^n\rtimes \S_n$ is induced by an affine map. If $(i_1\ldots i_m)$ is a non-trivial cycle of $\sigma_{\bar{z}}$ for some $\bar{z}$, then the images $\phi_{i_1}(\bar{z}),\ldots,\phi_{i_m}(\bar{z})$ cannot all be equal.
\end{cor}

\begin{proof}
Take $i$ so that the indices in the cycle can be written as $i,(\sigma_{\bar{z}})^{-1}(i),\ldots,(\sigma_{\bar{z}})^{-(n_{i\bar{z}}-1)}(i)$. Note that \[
\phi_i(n_{i\bar{z}}\bar{z})=\phi_i(\bar{z})+\phi_{(\sigma_{\bar{z}})^{-1}(i)}(\bar{z})+\ldots+\phi_{(\sigma_{\bar{z}})^{-(n_{i\bar{z}}-1)}(i)}(\bar{z}).
\]
Indeed, this is the $i$-th coordinate of $\psi(n_{i\bar{z}}\bar{z})$, on the one hand computed directly as $\psi(n_{i\bar{z}}\bar{z})$, and on the other hand as $\psi(\bar{z})^{n_{i\bar{z}}}$.
If all images $\phi_{i_1}(\bar{z}),\ldots,\phi_{i_m}(\bar{z})$ are equal, the right hand side equals $n_{i\bar{z}}\phi_i(\bar{z})$. It follows that $\phi_i(n_{i\bar{z}}\bar{z})\in n_{i\bar{z}}\Z^k$.
\end{proof}

\begin{ex}\label{ex:1}
For $n\geq 2$, consider the map
\begin{align*}
f:T^2\multimap T^2: p(t_1,t_2)\mapsto \{p(\tilde{f}_1(t_1,t_2)),\ldots,p(\tilde{f}_n(t_1,t_2)) \}
\end{align*}
with $\tilde{f}_i(t_1,t_2)=\left(\frac{1}{4}\cos\left(2\pi\left(\frac{t_1}{n}+\frac{i-1}{n} \right)\right),\frac{1}{4}\sin\left(2\pi\left(\frac{t_1}{n}+\frac{i-1}{n} \right)\right)\right)$.

For all $t_1,t_2\in \R$, and for all $i$, we have 
\begin{equation}\label{eq:ex1wd}
\begin{aligned}
\tilde{f}_i(t_1+1,t_2)&=\tilde{f}_{i+1\,\text{mod}\,n}(t_1,t_2) \\
\tilde{f}_i(t_1,t_2+1)&=\tilde{f}_i(t_1,t_2).
\end{aligned}
\end{equation}
It follows that $f$ is well-defined.
To prove that $f$ is $n$-valued, suppose $\tilde{f}_i(t_1,t_2)\equiv \tilde{f}_{j}(t_1,t_2) \text{ mod }\Z^2$ for some $i$ and $j$. Then in particular \[
\left\{
\begin{array}{l}
\frac{1}{4}\cos\big(2\pi\big(\frac{t_1}{n}+\frac{i-1}{n} \big)\big) \equiv \frac{1}{4}\cos\big(2\pi\big(\frac{t_1}{n}+\frac{j-1}{n} \big)\big) \text{ mod }\Z \\
\frac{1}{4}\sin\big(2\pi\big(\frac{t_1}{n}+\frac{i-1}{n} \big)\big) \equiv \frac{1}{4}\sin\big(2\pi\big(\frac{t_1}{n}+\frac{j-1}{n} \big)\big) \text{ mod }\Z.
\end{array}
\right.
\]
Since the left and right hand side of both equations are numbers between $-\frac{1}{4}$ and $\frac{1}{4}$, this can only happen if they are actually equal, so \[
\left\{
\begin{array}{l}
\cos\big(2\pi\big(\frac{t_1}{n}+\frac{i-1}{n} \big)\big) = \cos\big(2\pi\big(\frac{t_1}{n}+\frac{j-1}{n} \big)\big) \\
\sin\big(2\pi\big(\frac{t_1}{n}+\frac{i-1}{n} \big)\big) = \sin\big(2\pi\big(\frac{t_1}{n}+\frac{j-1}{n} \big)\big),
\end{array}
\right.
\]
which is the case if and only if the arguments on the left and right differ up to an integer multiple of $2\pi$.
It follows that $\frac{i-j}{n}\in \Z$, which cannot happen for distinct $i,j\in \{1,\ldots,n\}$.

The equations (\ref{eq:ex1wd}) imply that $\psi_{\tilde{f}}$ is given by \[
\psi_{\tilde{f}}(z_1,z_2)=((0,0),\ldots,(0,0);(n\cdots 2\hspace{1pt}1)^{z_1}).
\]
In particular, we have $\psi_{\tilde{f}}(1,0)=((0,0),\ldots,(0,0);(n\cdots2\hspace{1pt}1))$, which would be impossible if $\psi_{\tilde{f}}$ were induced by an affine map, by Corollary \ref{cor:phis-equal}. It follows that the given map $f$ cannot be homotopic to an affine map.

Notice that this example can easily be extended to an $n$-valued map on $T^k$ for $k>2$ by adding zeroes in the other coordinates:
\[
T^k\multimap T^k: p(t_1,\ldots,t_k)\mapsto \{p(\tilde{f}_1(t_1,t_2),0,\ldots,0),\ldots,p(\tilde{f}_n(t_1,t_2),0,\ldots,0) \}.
\]
The above arguments still apply to this new map, so this way we get a non-affine $n$-valued map on $T^k$ for any $n\geq 2$ and $k\geq 2$.
\end{ex}

In the next section we will see that, after a smart choice of reference lift, any obstruction to affineness of the type of Theorem \ref{thm:nec-cond} reduces to one of the type of Corollary \ref{cor:phis-equal}.

\section{The influence of the choice of lift}

For a given $n$-valued map $f:T^k\multimap T^k$, we can investigate to what extent the morphism $\psi_{\tilde{f}}$ depends on the lift $\tilde{f}$ chosen to compute it. A first result is that, up to the order of the lift-factors, the morphisms $\sigma$ and $\phi_i:S_i\to \Z^k$ are independent of the chosen lift.

\begin{lem}
Let $f:T^k\multimap T^k$ be an $n$-valued map with lifts $\tilde{f}=(\tilde{f}_1,\ldots,\tilde{f}_n)$ and $\tilde{f}'=(\tilde{f}'_1,\ldots,\tilde{f}'_n)$, respectively inducing morphisms $\psi_{\tilde{f}}=(\phi_1,\ldots,\phi_n;\sigma)$ and $\psi_{\tilde{f}'}=(\phi'_1,\ldots,\phi'_n;\sigma')$ with corresponding groups $S_i$ and $S_i'$. If the lift-factors of $\tilde{f}$ and $\tilde{f}'$ are ordered in the same way, i.e. \[
(p(\tilde{f}_1(\bar{0})),\ldots,p(\tilde{f}_n(\bar{0})))=(p(\tilde{f}'_1(\bar{0})),\ldots,p(\tilde{f}'_n(\bar{0}))),
\] 
then $\sigma=\sigma'$ (and therefore also $S_i=S_i'$ for all $i$) and the restrictions of $\phi_i$ and $\phi'_i$ to $S_i$ are the same for all $i$.
\end{lem}

\begin{proof}
For all $\bar{t}$, $\bar{z}$ and $i$, we have \begin{equation}\label{eq:psi-different-lifts}
\left\{
\begin{array}{l}
\tilde{f}_i(\bar{t}+\bar{z})=\tilde{f}_{(\sigma_{\bar{z}})^{-1}(i)}(\bar{t})+\phi_i(\bar{z}) \\ \tilde{f}'_i(\bar{t}+\bar{z})=\tilde{f}'_{(\sigma'_{\bar{z}})^{-1}(i)}(\bar{t})+\phi'_i(\bar{z}).
\end{array}
\right.
\end{equation}
By covering space theory (applied to the coverings $p: \R^k \to \Z^k$ and 
 $p^n: F_n(\R^k, \Z^k) \to D_n(T^k)$), there is a covering translation  $(\bar{z}_1,\ldots,\bar{z}_n;\tau)\in (\Z^k)^n\rtimes \S_n$  such that \[
(\tilde{f}'_1,\ldots,\tilde{f}'_n)=(\bar{z}_1,\ldots,\bar{z}_n;\tau)(\tilde{f}_1,\ldots,\tilde{f}_n)=(\tilde{f}_{\tau^{-1}(1)}+\bar{z}_1,\ldots,\tilde{f}_{\tau^{-1}(n)}+\bar{z}_n).
\]
Since we assumed $p(\tilde{f}_i(\bar{0}))=p(\tilde{f}'_i(\bar{0}))$ for all $i$, we have \[
p(\tilde{f}_i(\bar{0}))=p(\tilde{f}_{\tau^{-1}(i)}(\bar{0})+\bar{z}_i)=p(\tilde{f}_{\tau^{-1}(i)}(\bar{0}))
\]
for all $i$. Since all images $p(\tilde{f}_i(\bar{0}))$ are distinct, it follows that $\tau^{-1}(i)=i$ for all $i$, and hence $\tilde{f}'_i=\tilde{f}_i+\bar{z}_i$. Plugging this into the equation (\ref{eq:psi-different-lifts}) gives \[
\left\{
\begin{array}{l}
\tilde{f}_i(\bar{t}+\bar{z})=\tilde{f}_{(\sigma_{\bar{z}})^{-1}(i)}(\bar{t})+\phi_i(\bar{z}) \\ \tilde{f}_i(\bar{t}+\bar{z})+\bar{z}_i=\tilde{f}_{(\sigma'_{\bar{z}})^{-1}(i)}(\bar{t})+\bar{z}_{(\sigma'_{\bar{z}})^{-1}(i)}+\phi'_i(\bar{z}).
\end{array}
\right.
\]
In particular, \[
\left\{
\begin{array}{l}
p(\tilde{f}_i(\bar{t}+\bar{z}))=p(\tilde{f}_{(\sigma_{\bar{z}})^{-1}(i)}(\bar{t})+\phi_i(\bar{z}))=p(\tilde{f}_{(\sigma_{\bar{z}})^{-1}(i)}(\bar{t})) \\ p(\tilde{f}_i(\bar{t}+\bar{z}))=p(\tilde{f}_{(\sigma'_{\bar{z}})^{-1}(i)}(\bar{t})+\bar{z}_{(\sigma'_{\bar{z}})^{-1}(i)}+\phi'_i(\bar{z})-\bar{z}_i)=p(\tilde{f}_{(\sigma'_{\bar{z}})^{-1}(i)}(\bar{t})),
\end{array}
\right.
\]
so that $\sigma_{\bar{z}}=\sigma'_{\bar{z}}$ for all $\bar{z}$. On the other hand, for $\bar{z}\in S_i=S_i'$, \[
\left\{
\begin{array}{l}
\tilde{f}_i(\bar{t}+\bar{z})=\tilde{f}_i(\bar{t})+\phi_i(\bar{z}) \\ \tilde{f}_i(\bar{t}+\bar{z})+\bar{z}_i=\tilde{f}_i(\bar{t})+\bar{z}_i+\phi'_i(\bar{z}),
\end{array}
\right.
\]
so $\phi_i(\bar{z})=\phi'_i(\bar{z})$.
\end{proof}

In particular, to check whether a given $n$-valued map is homotopic to an affine map using Theorem \ref{thm:nec-cond}, it does not matter with respect to which lift one considers the induced morphism, since this condition only depends on the sets $S_i$ and the values of $\phi_i$ on $S_i$. 

On the other hand, outside the sets $S_i$, the maps $\phi_i$ can attain all possible values, as far as the previous condition allows:

\begin{lem}\label{lem:lift-choice}
Let $f:T^k\multimap T^k$ be an $n$-valued map, and $\psi_{\tilde{f}}=(\phi_1,\ldots,\phi_n;\sigma)$ the induced morphism with respect to some lift $\tilde{f}$. 
Fix $i\in \{1,\ldots,n\}$ and $\bar{z}\in \Z^k$. For any possible decomposition $\phi_i(n_{i\bar{z}}\bar{z})=\bar{z}_0+\bar{z}_1+\ldots+\bar{z}_{n_{i\bar{z}}-1}$ into a sum of $n_{i\bar{z}}$ integer vectors, there is another lift $\tilde{f}'$ of $f$ with induced morphism $\psi_{\tilde{f}'}=(\phi'_1,\ldots,\phi'_n;\sigma')$ such that $\phi'_{(\sigma'_{\bar{z}})^{-m}(i)}(\bar{z})=\bar{z}_m$ for all $m=0,\ldots,n_{i\bar{z}}-1$.
\end{lem}

\begin{proof}
For the lift $\tilde{f}$, we have \[
\tilde{f}_{(\sigma_{\bar{z}})^{-m}(i)}(\bar{t}+\bar{z})=\tilde{f}_{(\sigma_{\bar{z}})^{-(m+1)}(i)}(\bar{t})+\phi_{(\sigma_{\bar{z}})^{-m}(i)}(\bar{z})
\]
for all $m\in \Z$.
Define the lift $\tilde{f}'=(\tilde{f}'_1,\ldots,\tilde{f}'_n)$ by \[
\tilde{f}'_{(\sigma_{\bar{z}})^{-m}(i)}=\tilde{f}_{(\sigma_{\bar{z}})^{-m}(i)}+\sum_{j=0}^{m-1}(\phi_{(\sigma_{\bar{z}})^{-j}(i)}(\bar{z})-\bar{z}_j)
\]
for $m=0,\ldots,n_{i\bar{z}}-1$.
Then we have \begin{align*}
\tilde{f}'_{(\sigma_{\bar{z}})^{-m}(i)}(\bar{t}+\bar{z})&=\tilde{f}_{(\sigma_{\bar{z}})^{-m}(i)}(\bar{t}+\bar{z})+\sum_{j=0}^{m-1}(\phi_{(\sigma_{\bar{z}})^{-j}(i)}(\bar{z})-\bar{z}_j) \\
&=\tilde{f}_{(\sigma_{\bar{z}})^{-(m+1)}(i)}(\bar{t})+\phi_{(\sigma_{\bar{z}})^{-m}(i)}(\bar{z})+\sum_{j=0}^{m-1}(\phi_{(\sigma_{\bar{z}})^{-j}(i)}(\bar{z})-\bar{z}_j) \\
&=\tilde{f}_{(\sigma_{\bar{z}})^{-(m+1)}(i)}(\bar{t})+\sum_{j=0}^{m}(\phi_{(\sigma_{\bar{z}})^{-j}(i)}(\bar{z})-\bar{z}_j)+\bar{z}_m \\
&=\tilde{f}'_{(\sigma_{\bar{z}})^{-(m+1)}(i)}(\bar{t})+\bar{z}_m
\end{align*}
for all $m=0,\ldots,n_{i\bar{z}}-2$; and for $m=n_{i\bar{z}}-1$ we get \begin{align*}
\tilde{f}'_{(\sigma_{\bar{z}})^{-(n_{i\bar{z}}-1)}(i)}(\bar{t}+\bar{z})&=\tilde{f}_{(\sigma_{\bar{z}})^{-(n_{i\bar{z}}-1)}(i)}(\bar{t}+\bar{z})+\sum_{j=0}^{n_{i\bar{z}}-2}(\phi_{(\sigma_{\bar{z}})^{-j}(i)}(\bar{z})-\bar{z}_j) \\
&=\tilde{f}_{i}(\bar{t})+\phi_{(\sigma_{\bar{z}})^{-(n_{i\bar{z}}-1)}(i)}(\bar{z})+\sum_{j=0}^{n_{i\bar{z}}-2}(\phi_{(\sigma_{\bar{z}})^{-j}(i)}(\bar{z})-\bar{z}_j) \\
&=\tilde{f}_{i}(\bar{t})+\sum_{j=0}^{n_{i\bar{z}}-1}(\phi_{(\sigma_{\bar{z}})^{-j}(i)}(\bar{z})-\bar{z}_j)+\bar{z}_{n_{i\bar{z}}-1} \\
&=\tilde{f}_{i}(\bar{t})+\bar{z}_{n_{i\bar{z}}-1} \\
&=\tilde{f}'_i(\bar{t})+\bar{z}_{n_{i\bar{z}}-1},
\end{align*}
since \[
\sum_{j=0}^{n_{i\bar{z}}-1}(\phi_{(\sigma_{\bar{z}})^{-j}(i)}(\bar{z})-\bar{z}_j)=\sum_{j=0}^{n_{i\bar{z}}-1}\phi_{(\sigma_{\bar{z}})^{-j}(i)}(\bar{z})-\sum_{j=0}^{n_{i\bar{z}}-1}\bar{z}_j=\phi_i(n_{i\bar{z}}\bar{z})-\sum_{j=0}^{n_{i\bar{z}}-1}\bar{z}_j
\]
which is zero by assumption.
\end{proof}

In particular, for a non-affine $n$-valued map $f$, the condition of Corollary \ref{cor:phis-equal} that the images $\phi_i(\bar{z})$ are equal on a cycle of $\sigma_{\bar{z}}$ will not be satisfied for the induced morphisms of \emph{all} lifts of $f$, but there will always be a lift whose induced morphism satisfies this condition: by Theorem \ref{thm:nec-cond}, there are $i$, $\bar{z}\notin S_i$ and $\bar{z}'\in \Z^k$ such that $\phi_i(n_{i\bar{z}}\bar{z})=n_{i\bar{z}}\bar{z}'$, so taking the decomposition $\phi_i(n_{i\bar{z}}\bar{z})=\bar{z}'+\ldots+\bar{z}'$ in the above lemma yields a lift with respect to which all images $\phi_i(\bar{z}),\phi_{(\sigma_{\bar{z}})^{-1}(i)}(\bar{z}),\ldots,\phi_{(\sigma_{\bar{z}})^{-(n_{i\bar{z}}-1)}(i)}(\bar{z})$ are equal to $\bar{z}'$.

A special case is the one where $\bar{z}'=\bar{0}$, as in Example \ref{ex:1}. It turns out that this is the case whenever the image of $\psi$ contains torsion:

\begin{lem}
If the image of a morphism $\psi=(\phi_1,\ldots,\phi_n;\sigma):\Z^k\to\Z^k\rtimes \S_n$ contains torsion, then $\phi_i(n_{i\bar{z}}\bar{z})=\bar{0}$ for some $\bar{z}$ and $i$ such that $\bar{z}\notin S_i$.
\end{lem}

\begin{proof}
Suppose there are $\bar{z}\in \Z^k$ and $m\in \Z_{>0}$ such that $\psi(\bar{z})\neq 1$ but $\psi(m\bar{z})=\psi(\bar{z})^m=1$. Then $m$ is a multiple of $n_{i\bar{z}}$ for all $i$, since $\psi(\bar{z})^m=1$ implies $\sigma_{\bar{z}}^m=1$. Write $m=en_{i\bar{z}}$ with $e\in \Z$. Since $n_{i\bar{z}}\bar{z}\in S_i$, we have $e\phi_i(n_{i\bar{z}}\bar{z})=\phi_i(m\bar{z})=\bar{0}$, from which it follows that $\phi_i(n_{i\bar{z}}\bar{z})=\bar{0}$. 

Note that $\bar{z}\notin S_i$ for at least one value of $i$, for otherwise $n_{i\bar{z}}=1$ for all $i$ and hence $\sigma_{\bar{z}}=1$ and $\phi_i(\bar{z})=\bar{0}$ for all $i$, i.e. $\psi(\bar{z})=1$.
\end{proof}

Thus, if $f$ is an $n$-valued map with lift $\tilde{f}$ such that the image of $\psi_{\tilde{f}}$ contains torsion, there is always another lift for which all images $\phi_i(\bar{z})$ are zero on a non-trivial cycle of $\sigma_{\bar{z}}$, for some $\bar{z}$. In particular, by Corollary \ref{cor:phis-equal},

\begin{cor}
If $f:T^k\to D_n(T^k)$ is affine, the image of $\psi_{\tilde{f}}$ is torsion free, for any lift $\tilde{f}$.
\end{cor}

It can also happen that $\bar{z}'\neq \bar{0}$, so the absence of torsion in the image of $\psi_{\tilde{f}}$ is not sufficient for being induced by an affine map:

\begin{ex}\label{ex:2}
We can slightly adapt our first example by setting \[
\textstyle
\tilde{f}_i(t_1,t_2)=\left(t_1+\frac{1}{4}\cos\left(2\pi\left(\frac{t_1}{n}+\frac{i-1}{n} \right)\right),\frac{1}{4}\sin\left(2\pi\left(\frac{t_1}{n}+\frac{i-1}{n} \right)\right)\right).
\]
For all $t_1,t_2\in \R$, and for all $i$, we have 
\begin{equation}\label{eq:ex3wd}
\begin{aligned}
\tilde{f}_i(t_1+1,t_2)&=\tilde{f}_{i+1\,\text{mod}\,n}(t_1,t_2)+(1,0) \\
\tilde{f}_i(t_1,t_2+1)&=\tilde{f}_i(t_1,t_2),
\end{aligned}
\end{equation}
so $f$ is well-defined. To show that it is $n$-valued, we can use exactly the same argument as in the first example.

It follows from (\ref{eq:ex3wd}) that $\psi_{\tilde{f}}$ is given by \[
\psi_{\tilde{f}}(z_1,z_2)=((z_1,0),\ldots,(z_1,0);(n\cdots 2\hspace{1pt}1)^{z_1}).
\]
In particular, we have $\psi_{\tilde{f}}(1,0)=((1,0),\ldots,(1,0);(n\cdots2\hspace{1pt}1))$, which would be impossible if $\psi_{\tilde{f}}$ were induced by an affine map, by Corollary \ref{cor:phis-equal}. However, the image of $\psi_{\tilde{f}}$ contains no torsion, since for all $z_1,z_2$ and $m$, \[
\psi_{\tilde{f}}(z_1,z_2)^m=((mz_1,0),\ldots,(mz_1,0);(n\cdots 2\hspace{1pt}1)^{mz_1}),
\] which is trivial if and only if $z_1=0$, and in that case $\psi_{\tilde{f}}(z_1,z_2)$ itself is trivial.
\end{ex}

\section{Constructing more complicated non-affine maps}

In the previous example, we constructed an $n$-valued map with non-trivial maps $\phi_i$ starting from one with trivial maps $\phi_i$ with the same morphism $\sigma$. This can be done more generally:

\begin{lem}\label{lem:gen-psi-from-torsion}
Suppose $f:T^k\multimap T^k$ is an irreducible $n$-valued map with lift $\tilde{f}=(\tilde{f}_1,\ldots,\tilde{f}_n)$ such that $\psi_{\tilde{f}}=(\bar{0},\ldots,\bar{0};\sigma)$, and let $\psi=(\phi_1,\ldots,\phi_n;\sigma):\Z^k\to (\Z^k)^n\rtimes \S_n$ be a morphism with the same map $\sigma$ and arbitrary maps $\phi_i$. Then there are $A\in \R^{k\times k}$, $\bar{a}_1,\ldots,\bar{a}_n\in \R^k$ and $\varepsilon>0$ such that \[
p(\bar{t})\mapsto (p(A\bar{t}+\bar{a}_1+\varepsilon\tilde{f}_1(\bar{t})),\ldots,p(A\bar{t}+\bar{a}_n+\varepsilon\tilde{f}_n(\bar{t})))
\]
defines an $n$-valued map inducing the morphism $\psi$ with respect to the lift \[
\bar{t}\mapsto (A\bar{t}+\bar{a}_1+\varepsilon\tilde{f}_1(\bar{t}),\ldots,A\bar{t}+\bar{a}_n+\varepsilon\tilde{f}_n(\bar{t})).
\]
\end{lem}

\begin{proof}
Let $\bar{e}_1,\ldots,\bar{e}_k$ be the standard basis of $\Z^k$, and denote $\sigma_{\bar{e}_j}=\vcentcolon\sigma_j$ and $n_{\bar{e}_j}=\vcentcolon n_j$ for all $j\in \{1,\ldots,k\}$. Then $\psi_{\tilde{f}}(n_j\bar{e}_j)=1$, so $\psi_{\tilde{f}}(mn_j\bar{e}_j)=1$ for any $m\in \Z$, which means that $\tilde{f}_i(\bar{t}+mn_j\bar{e}_j)=\tilde{f}_i(\bar{t})$ for all $m$, $i$ and $\bar{t}$. Since the maps $\tilde{f}_i$ are continuous and the set $\{mn_j\bar{e}_j\mid m\in \Z,\,j=1,\ldots,k \}$ is a lattice in $\R^k$, it follows that the maps $\tilde{f}_i$ are bounded; say $M\in \R_{>0}$ is such that $|(\tilde{f}_i(\bar{t}))_j|<M$ for all $i,j$ and $\bar{t}$, where $(\tilde{f}_i(\bar{t}))_j$ denotes the $j$-th component of $\tilde{f}_i(\bar{t})$. Define $\varepsilon\vcentcolon=1/(2M\max_j n_j)$.

Let the matrix $A$ and the points $\bar{a}_1,\ldots,\bar{a}_n$ be as in the proof of Theorem \ref{thm:suff-cond}, where it is shown that they are well-defined and satisfty \[
A\bar{z}+\bar{a}_i=\bar{a}_{(\sigma_{\bar{z}})^{-1}(i)}+\phi_i(\bar{z})
\]
for all $i$ and $\bar{z}$. Since $A\,n_j\bar{e}_j\in \Z^k$ for all $j$, the $j$-th column of $A$ consists of numbers in $\frac{1}{n_j}\Z^k$. Therefore, we have $(A\bar{z})_j\in \frac{1}{n_j}\Z^k$ for all $j$ and for all $\bar{z}$.
By construction of the points $\bar{a}_i$, it follows that $(\bar{a}_i)_j\in \frac{1}{n_j}\Z^k$ for all $i$ and $j$.

With these $A$, $\bar{a}_i$ and $\varepsilon$, the map in the statement is well-defined and the given lift induces $\psi$, since \[
A(\bar{t}+\bar{z})+\bar{a}_i+\varepsilon\tilde{f}_i(\bar{t}+\bar{z})=A\bar{t}+\bar{a}_{(\sigma_{\bar{z}})^{-1}(i)}+\phi_i(\bar{z})+\varepsilon\tilde{f}_{(\sigma_{\bar{z}})^{-1}(i)}(\bar{t})
\]
for all $i$ and $\bar{z}$. To see that the map is $n$-valued, suppose \[
A\bar{t}+\bar{a}_i+\varepsilon\tilde{f}_i(\bar{t})\equiv A\bar{t}+\bar{a}_{i'}+\varepsilon\tilde{f}_{i'}(\bar{t}) \text{ mod }\Z^k
\]
for some indices $i,i'\in \{1,\ldots,n\}$ and for some $\bar{t}$. Then \[
\varepsilon(\tilde{f}_{i'}(\bar{t})-\tilde{f}_i(\bar{t}))\equiv \bar{a}_i-\bar{a}_{i'} \text{ mod }\Z^k.
\]
By definition of $\varepsilon$, the $j$-th component of the left hand side lies in $(-\frac{1}{n_j},\frac{1}{n_j})$ whereas the $j$-th component of the right hand side lies in $\frac{1}{n_j}\Z^k$, which can only happen if both sides are zero. Thus, \[
\varepsilon\tilde{f}_{i'}(\bar{t})=\varepsilon\tilde{f}_i(\bar{t}).
\]
But this implies that $\tilde{f}_{i'}(\bar{t})=\tilde{f}_i(\bar{t})$, and hence $i=i'$ because the given map $f$ was $n$-valued.
\end{proof}

Thus, if we want to find examples of $n$-valued maps inducing given morphisms $\psi$, it suffices to restrict our attention to the morphisms $\psi$ for which all maps $\phi_i$ are trivial. Below, we show two ways in which the idea of Example \ref{ex:1} can be extended to more complicated situations of this type.

\begin{ex}
Consider the morphism $\psi:\Z^2\to (\Z^2)^4\rtimes \S_4$ defined by \begin{align*} 
\psi(\bar{e}_1)&=(\bar{0},\bar{0},\bar{0},\bar{0};(12)(34)) \\
\psi(\bar{e}_2)&=(\bar{0},\bar{0},\bar{0},\bar{0};(13)(24)).
\end{align*}
A $4$-valued map inducing this morphism is given by the lift-factors \begin{align*}
\tilde{f}_1(t_1,t_2)&=\textstyle(\frac{1}{4}\cos(2\pi\frac{t_1}{2})+\frac{1}{8}\cos(2\pi\frac{t_2}{2}),\frac{1}{4}\sin(2\pi\frac{t_1}{2})+\frac{1}{8}\sin(2\pi\frac{t_2}{2})) \\
\tilde{f}_2(t_1,t_2)&=\textstyle(\frac{1}{4}\cos(2\pi\frac{t_1+1}{2})+\frac{1}{8}\cos(2\pi\frac{t_2}{2}),\frac{1}{4}\sin(2\pi\frac{t_1+1}{2})+\frac{1}{8}\sin(2\pi\frac{t_2}{2})) \\
\tilde{f}_3(t_1,t_2)&=\textstyle(\frac{1}{4}\cos(2\pi\frac{t_1}{2})+\frac{1}{8}\cos(2\pi\frac{t_2+1}{2}),\frac{1}{4}\sin(2\pi\frac{t_1}{2})+\frac{1}{8}\sin(2\pi\frac{t_2+1}{2})) \\
\tilde{f}_4(t_1,t_2)&=\textstyle(\frac{1}{4}\cos(2\pi\frac{t_1+1}{2})+\frac{1}{8}\cos(2\pi\frac{t_2+1}{2}),\frac{1}{4}\sin(2\pi\frac{t_1+1}{2})+\frac{1}{8}\sin(2\pi\frac{t_2+1}{2})).
\end{align*}
Since the values of all components of these lift-factors lie between $-\frac{3}{8}$ and $\frac{3}{8}$, two of them are the same mod $\Z^2$ if and only if they are the same. The fact that $\tilde{f}_1(\bar{t})\neq \tilde{f}_2(\bar{t})$, $\tilde{f}_1(\bar{t})\neq \tilde{f}_3(\bar{t})$, $\tilde{f}_2(\bar{t})\neq \tilde{f}_4(\bar{t})$ and $\tilde{f}_3(\bar{t})\neq \tilde{f}_4(\bar{t})$ for all $\bar{t}$ follows by the argument of Example \ref{ex:1}, because each time one of the terms on both sides cancels out. For $\tilde{f}_1(\bar{t})\neq \tilde{f}_4(\bar{t})$ and $\tilde{f}_2(\bar{t})\neq \tilde{f}_3(\bar{t})$, it suffices to observe that it never happens at the same time that \[
\left\{
\begin{array}{l}
\cos(2\pi \frac{t_1}{2})-\cos(2\pi \frac{t_1+1}{2})\in [-1,1] \\
\sin(2\pi \frac{t_1}{2})-\sin(2\pi \frac{t_1+1}{2})\in [-1,1]
\end{array}
\right.
\] 
so that it cannot simultaneously happen that \[
\left\{
\begin{array}{l}
\cos(2\pi \frac{t_1}{2})-\cos(2\pi \frac{t_1+1}{2})=\pm(\frac{1}{2}\cos(2\pi \frac{t_2}{2})-\frac{1}{2}\cos(2\pi \frac{t_2+1}{2})) \\
\sin(2\pi \frac{t_1}{2})-\sin(2\pi \frac{t_1+1}{2})=\pm(\frac{1}{2}\sin(2\pi \frac{t_2}{2})-\frac{1}{2}\sin(2\pi \frac{t_2+1}{2})).
\end{array}
\right.
\]
\end{ex}

\smallskip

\begin{ex}
The morphism $\psi:\Z^2\to (\Z^2)^4\rtimes \S_4$ defined by \begin{align*} 
\psi(\bar{e}_1)&=(\bar{0},\bar{0},\bar{0},\bar{0};(1234)) \\
\psi(\bar{e}_2)&=(\bar{0},\bar{0},\bar{0},\bar{0};(13)(24))
\end{align*}
is induced by the $4$-valued map with lift-factors \begin{align*}
\tilde{f}_1(t_1,t_2)&=\textstyle(\frac{1}{4}\cos(2\pi\frac{t_1+2t_2}{4}),\frac{1}{4}\sin(2\pi\frac{t_1+2t_2}{4})) \\
\tilde{f}_2(t_1,t_2)&=\textstyle(\frac{1}{4}\cos(2\pi\frac{t_1+2t_2+1}{4}),\frac{1}{4}\sin(2\pi\frac{t_1+2t_2+1}{4})) \\
\tilde{f}_1(t_1,t_2)&=\textstyle(\frac{1}{4}\cos(2\pi\frac{t_1+2t_2+2}{4}),\frac{1}{4}\sin(2\pi\frac{t_1+2t_2+2}{4})) \\
\tilde{f}_2(t_1,t_2)&=\textstyle(\frac{1}{4}\cos(2\pi\frac{t_1+2t_2+3}{4}),\frac{1}{4}\sin(2\pi\frac{t_1+2t_2+3}{4})).
\end{align*}
Here, $4$-valuedness just follows by the argument of Example \ref{ex:1}.
\end{ex}


\end{document}